\documentclass[12pt, draft]{amsart} 

\usepackage[utf8]{inputenc}	
\usepackage{mathtools, color, enumerate, hyperref, mathdots, bookmark}
\usepackage[capitalize]{cleveref}

\theoremstyle{definition}
\newtheorem{theorem}{Theorem}[section]
\newtheorem{lemma}[theorem]{Lemma}
\newtheorem{corollary}[theorem]{Corollary}
\newtheorem{definition}[theorem]{Definition}

\newtheorem{example}[theorem]{Example}

\DeclareMathOperator{\Inn}{Inn}
\DeclareMathOperator{\Aut}{Aut}

\DeclareMathOperator{\GL}{GL}
\DeclareMathOperator{\SL}{SL}
\DeclareMathOperator{\SO}{SO}
\DeclareMathOperator{\Sp}{Sp}
\DeclareMathOperator{\Id}{Id}

\DeclareMathOperator{\ch}{Char}
\DeclareMathOperator{\diag}{Diag}
\DeclareMathOperator{\Orth}{O}

\newcommand{\bbr}{\mathbb{R}}
\newcommand{\bbc}{\mathbb{C}}
\newcommand{\bbf}{\mathbb{F}}

\newcommand{\ol}[1]{\overline{#1}}

\newcommand{\st}{\ |\ }

\newcommand{\Lx}[1]{\begin{bsmallmatrix} 0 & 1\\ #1 & 0 \end{bsmallmatrix}}
\newcommand{\Ux}[1]{\begin{bsmallmatrix} 1 & 0\\#1 & 1\end{bsmallmatrix}}
\newcommand{\slnk}{\SL(n,k)}

\newcommand{\lnp}{L_{\frac{n}{2},p}}
\newcommand{\lmc}{L_{m,c^2,c}}

\title{$k$-involutions of $\slnk$ over Fields of Characteristic 2}
\author{Nathaniel J. Schwartz}
\date{\today}
\thanks{The author thanks A. Helminck for his help and guidance.}

\begin{document}

\begin{abstract}
Symmetric $k$-varieties generalize Riemannian sym\-me\-tric spaces to reductive groups defined over arbitrary fields. For most perfect fields, it is known that symmetric $k$-varieties are in one-to-one correspondence with isomorphy classes of $k$-involutions. Therefore, it is useful to have representatives of each isomorphy class in order to describe the $k$-varieties. Here we give matrix representatives for each isomorphy class of $k$-involutions of $\slnk$ in the case that $k$ is any field of characteristic 2; we also describe fixed point groups of each type of involution.
\end{abstract}

\maketitle

\section{Introduction}
Let $G$ be a connected reductive linear algebraic group defined over a field $k$, and let $\varphi$ be an involution of $G$, $H$ a $k$-open subgroup of the fixed-point group of $\varphi$, and let $G_k$ (resp. $H_k$) be the set of $k$-rational points of $G$ (resp. $H$). The homogeneous space $G_k/H_k$ is called a symmetric $k$-variety. Symmetric $k$-varieties generalize the notion of Riemannian symmetric spaces of Lie groups defined over $\bbr$ or $\bbc$ to reductive linear algebraic groups defined over arbitrary fields $k$ which may be non-algebraically closed. Symmetric $k$-varieties are important in representation theory and they arise in a variety of other areas.

In 1994, Helminck described isomorphy classes of symmetric $k$-va\-ri\-e\-ties in \cite{hel94}. One of the significant results he states is that isomorphy classes of symmetric $k$-varieties are in one to one correspondence with isomorphy classes of $k$-involutions of the associated algebraic groups. In the same paper, Helminck also gave a characterization of isomorphism classes of $k$-involutions which uses the following three invariants:
\begin{enumerate}
\item Classification of admissible $(\Gamma, \varphi)$-indices.
\item Classification of the $G_k$-isomorphism classes of $k$-involutions of the $k$-anisotropic kernel of $G$.
\item Classification of the $G_k$-isomorphism classes of $k$-inner elements of $G$.
\end{enumerate}
In 1998, Ling Wu classified and characterized the equivalence classes of $k$-involutions of $\slnk$ for algebraically closed fields, $\mathfrak{p}$-adic fields, and finite fields of odd characteristic (see \cite{wu02}). Here, we present a similar classification and characterization in the case that our field $k$ is any field of characteristic 2. 

Following is a brief overview of the contents of this paper. We begin by determining, up to isomorphism, involutions of $\SL(2,k)$. In this case every automorphisms is inner, given by conjugation by an element of $\GL(2,k)$, so we give explicit matrix representatives for each conjugacy class. We compute and discuss the structure of the fixed point groups. Next, we extend these methods to the case that $n \ge 2$ and determine matrix representatives for both inner and outer involutions. We conclude with a description of the fixed point groups of each representative involution and a description of the related symmetric $k$-varieties in each case.

There are several notable differences between the characteristic 2 case and results found by Wu. There are different types of involutions in the characteristic 2 case than otherwise, in both the inner and the outer case. Also, in characteristic 2, the matrix representatives of involutions are unipotent and are never semisimple, whereas the involutions in \cite{wu02} are always semisimple.

This project was inspired by the work of Wu in \cite{wu02} and \cite{helwu02} where he determined involutions of $\SL(n,k)$ for fields of characteristic not 2, as well as for the group $\SO(2n+1,k)$. This last was expanded on and improved by Benim and others in \cite{ben13}. More recently, Benim, Helminck, and Ward published similar results for the groups $\Sp(2n,k)$ in \cite{ben15}.

\section{Notation and Definitions}
Unless otherwise specified, $k$ always represents a field of characteristic 2, $\ol{k}$ is the algebraic closure of $k$, and  $K$ is some extension field of $k$; thus $k \subset K \subseteq \ol{k}$. We make the following conventions throughout: $G = \SL(n, \ol{k})$, $G_k = \SL(n,k)$, and $G_K = \SL(n, K)$. 

For any $g \in G$, the map $\Inn_g(x) = gxg^{-1}$ is called an \emph{inner automorphism} of $G$; additionally, some elements of $\GL(n,k)$ or $\GL(n,K)$ induce inner automorphisms of $G$. The set of inner automorphisms is denoted by $\Inn(G_k)$. An automorphism is called {\em outer} if it is not inner. Denote by $\Aut(G_k)$ the inner automorphisms of $G_k$ induced by elements of $\GL(n,k)$. Then $\Aut(G, G_k)$ denotes the group of automorphisms in $\Aut(G)$ that fix $G_k$; these are called \emph{$k$-automorphisms}. Similar notation applies to inner automorphisms. In general, an {\em involution} in $\Aut(G)$ is an element of order 2. A $k$-involution of $G_k$ is an involution in $\Aut(G, G_k)$. At times we may simply use the term involution synonymously with $k$-involution.

Our goal is to determine all $k$-involutions in $G_k$. In general, there may exist both inner and outer automorphisms. For the inner case, the process is as follows. We first determine which automorphisms of $G_K$  fix $G_k$ point-wise. Separately, we determine which automorphisms of $G_K$ fix $G_k$ as a set. Finally, we combine this information to find which automorphisms have order two and fix $G_k$ point-wise. For the outer case, every outer automorphism is the product of an inner automorphism and a fixed outer automorphism, which for convenience we take to be the duality automorphism, except when $n=2$ and there are no outer automorphisms.

\begin{definition}
Let $\theta$ and $\varphi \in \Aut(G,G_k)$. If $\sigma\theta\sigma^{-1} = \varphi$ and
\begin{itemize}
      \item $\sigma \in \Aut(G,G_k)$, then $\theta$ and $\varphi$ are $\Aut(G,G_k)$-isomorphic.
      \item $\sigma \in \Inn(G, G_k)$, then $\theta$ and $\varphi$ are $\Inn(G,G_k)$-isomorphic.
\end{itemize}
\end{definition}
We want to determine which $k$-involutions are $\Inn(G,G_k)$-iso\-mor\-phic, and we will exhibit a list of representative elements for each class. Before we proceed, we note the following fact about fields of characteristic 2.

\begin{lemma}[J-P. Serre, \cite{ser73}]\label{serrelemma}
Let $k = \bbf_{2^r}$. Let $k^*$ denote the set of nonzero elements of $k$. Every element of $k^*$ is a square. Hence $|k^*/(k^*)^2| = 1$.
\end{lemma}
\section{$k$-involutions of $\SL(2,k)$}
In this section, we assume that $n = 2$, and we note the following consequence due to Borel (see Proposition 14.9 of \cite{bor91}).

\begin{lemma}\label{borel}
If $G = \SL(n,\ol{k})$, then $\Aut(G) = \Inn(G)$.
\end{lemma}

It follows that every $\varphi \in \Aut(G, G_k)$ is inner. Specifically, there is some $A \in \GL(2,K)$ such that $\varphi = \Inn_A|_{G_k}$. Notice that entries of $A$ may be in some intermediate extension field $K$ of $k$.

\begin{lemma}\label{fixG}
Suppose $A \in \GL(2,K)$. Then $\Inn_A|_{G_k} = \Id$ if and only if $A = p\Id$ for some $0\ne p \in K$.
\end{lemma}
\begin{proof}
Let $A = (a_{ij})$ with $i, j \in \{1, 2\}$ and $a_{ij} \in K$. Since $\Inn_A|_{G_k} = \Id$, then for all $X = (x_{ij}) \in G_k$, $\Inn_A(X) = AXA^{-1} = X$, so $A \in Z_{\GL(2,K)}(G_k)$. Direct computation shows that choosing $X_1 = \begin{bsmallmatrix} 1 & 0\\ 1 & 1 \end{bsmallmatrix}$ gives $a_{11} = a_{22}$ and $a_{12} = 0$ and choosing $X_2 = \begin{bsmallmatrix} 1 & 1\\ 0 & 1 \end{bsmallmatrix}$ gives $a_{21} = 0$.

The converse is clear.
\end{proof}

\begin{lemma}\label{keepInv}
Let $A \in \GL(2,K)$. Then $\Inn_A|_{G_k} \in \Aut(G_k)$ if and only if $A = pB$, for some $p \in K$ and $B \in \GL(2,k)$.
\end{lemma}

\begin{proof}
If $\Inn_A$ is invariant on $G_k$, then $\Inn_A(X) \in G_k$ for all $X \in G_k$. In particular, for the matrices $X_{ij}$
\[
X_{11} = \begin{bmatrix}0&1\\1&1\end{bmatrix} \hspace{.75cm} X_{12} = \begin{bmatrix}1&0\\1&1\end{bmatrix} \hspace{.75cm} X_{21} = \begin{bmatrix}1&1\\0&1\end{bmatrix} \hspace{.75cm} X_{22} = \begin{bmatrix}1&1\\1&0\end{bmatrix},
\]
and $Y_i$
\[
Y_1 = \begin{bmatrix}0&1\\1&0\end{bmatrix} \hspace{.75cm} Y_2 = \begin{bmatrix}1&0\\0&1\end{bmatrix},
\]
the entries of $\Inn_A(X_{ij})$ and $\Inn_A(Y_i)$ are elements of $k$. The sum $\Inn_A(X_{ij}) + \Inn_A(Y_i)$ reveals that 
\[
\frac{a_{ij}a_{k\ell}}{a_{11}a_{22} + a_{12}a_{21}} \in k
\]
for all $i, j, k$ and $\ell$. Fix $i$ and $j$ so that $a_{ij} \neq 0$, and set $r, s \in k$ as follows:
\[
r =  \frac{a_{ij}a_{k\ell}}{a_{11}a_{22} + a_{12}a_{21}} \hspace{.75cm}
s = \left(\frac{a_{11}a_{22} + a_{12}a_{21}}{a_{ij}a_{ij}}\right).
\] 
Thus $rs = \frac{a_{k\ell}}{a_{ij}} \in k$ for all $k$ and $\ell$, and hence
\[
A = a_{ij}
\begin{bmatrix}
\frac{a_{11}}{a_{ij}} & \frac{a_{12}}{a_{ij}}\\[5pt]
\frac{a_{21}}{a_{ij}} & \frac{a_{22}}{a_{ij}}
\end{bmatrix},
\]
where $a_{ij} \in K$. That is, $A = pB$ with $B \in \GL(2,k)$ and $p \in K$. (In fact $B \in G_k$ if $a_{ij}^{2} = 1$.)

The converse follows from the fact that $\Inn_{pB}(X) = pBX(pB)^{-1} = BXB^{-1}$, and $\det(BXB^{-1}) = \det(B)\det(X)\det(B^{-1}) = \det(X) = 1$.
\end{proof}

\begin{corollary}\label{ch2cor1}
Every automorphism of $G_k$ is of the form $\Inn_A$ for $A \in \GL(2,k)$.
\end{corollary}

Let $\varphi \in \Aut(G,G_k)$. Then $\varphi$ is inner, $\varphi^2(X) = (\Inn_A)^2(X) = A^2 X (A^2)^{-1}$, so by \cref{borel} and \cref{ch2cor1}
\[
\varphi^2 = (\Inn_A)^2 = \Inn_{A^2}
\]
for some $A \in \GL(2,k)$. Since $A \in \GL(2,k)$, by \cref{fixG}, $A^2 = p\Id$. Since $A \in \GL(2,k)$, then $p \in k$. In other words, if $\varphi$ is an involution, then $A^2 = p\Id$ with $0 \ne p \in k$. 

\begin{lemma}
Suppose $\varphi \in \Aut(G,G_k)$ is an involution. There exists some $0 \ne p \in k$ and a matrix $A \in \GL(2,k)$ such that $\varphi = \Inn_A|_{G_k}$, and $A$ is conjugate to $\begin{bsmallmatrix} 0 & 1\\p & 0\end{bsmallmatrix}$. Conversely, if $B \in \GL(2,K)$ is conjugate to $\begin{bsmallmatrix} 0 & 1\\p & 0\end{bsmallmatrix}$, then $\Inn_B \in \Aut(G,G_k)$.
\end{lemma}

\begin{proof}
Since $\varphi$ is an involution, there exists some $A \in \GL(2,k)$ such that $(\Inn_A)^2 = \Id$. Let 
\[
A = 
\begin{bmatrix}
a & b\\
c & d
\end{bmatrix},
\]
with $a, b, c, d \in k$.
Then 
\[
A^2 = 
\begin{bmatrix}
a^2 + bc & (a + d)b\\
(a + d)c & d^2 + bc
\end{bmatrix}.
\]
By remarks following \cref{ch2cor1}, $A^2 = p\Id$ for some $p \in k$, so 
\begin{align*}
a^2 + bc &= d^2 + bc\\
(a+d)c &= (a + d)b = 0.
\end{align*}
Since $a^2 = d^2$ implies $(a + d)^2 = 0$, it follows that $a = d$, and also $bc \ne a^2$ since $a^2 + bc \ne 0$. Then
$$A = \begin{bmatrix} a & b\\c & a\end{bmatrix}.$$
The characteristic polynomial is $C_A(x) = x^2 + a^2 + bc$. If the minimal polynomial is $M_A(x) = x + q$, for some $q \in k$, then $A = q\Id$ and $\Inn_A$ has order 1 and is not an involution. So $M_A(x) = C_A(x)$, and $A$ is conjugate to the (normalized) matrix
\[
B =
\begin{bmatrix}
0 & 1\\
p & 0
\end{bmatrix},
\]
where $p = (bc + a^2)^{-1} \in k$.

To show the converse, let $A = \begin{bsmallmatrix} 0 & 1\\p & 0\end{bsmallmatrix}$ with $p \in k$, and suppose $M \in GL(2,K)$ such that $B = M^{-1}AM$. Then $B^2 = M^{-1}A^2M$, and since $\Inn_M(X) \in G_k$ and $\Inn_A$ is a $k$-involution, we have
\begin{align*}
\Inn_{B^2}(X) &= (M^{-1}A^2M) X (M^{-1}A^2M)^{-1}\\
&= M^{-1}\big(A^2(MXM^{-1})A^{-2}\big)M \\
&= M^{-1}\left[\Inn_A^2 (MXM^{-1})\right] M\\
&= M^{-1}(MXM^{-1})M\\
&= X
\end{align*}
\end{proof}

\begin{example}
Since the characteristic and minimal polynomials of $A = \begin{bsmallmatrix}0 & 1\\p & 0\end{bsmallmatrix}$ are different, $A$ is never semisimple. However, $A$ can be unipotent. In the case that $k = \bbf_2$, there are three involutions corresponding to the matrices $\begin{bsmallmatrix}1 & 1\\ 0 & 1\end{bsmallmatrix}, \begin{bsmallmatrix}1 & 0\\ 1 & 1\end{bsmallmatrix}$, and $\begin{bsmallmatrix}0&1\\1&0\end{bsmallmatrix}$ (which are all conjugate). 
\end{example}

\begin{corollary}
Each isomorphism class of $k$-involutions of $G_k$ has representative $\Inn_A$, where $A \in \GL(2,k)$ is of the form
$$A = \Lx{p}.$$
\end{corollary}

\begin{lemma}\label{cacinv}
Let $\Inn_A$ and $\Inn_B$ be automorphisms in $\Aut(G,G_k)$. Then $\Inn_A$ and $\Inn_B$ are isomorphic if and only if there is a matrix $C \in \GL(2,k)$ and a constant $p \in k$ such that $CAC^{-1} = pB$.
\end{lemma}

\begin{proof}
By definition, $\Inn_A$ is isomorphic to $\Inn_B$ if and only if there is a matrix $C \in \GL(2,k)$ such that $\Inn_C\Inn_A\Inn_C^{-1} = \Inn_B$. Note that $\Inn_g\Inn_h = \Inn_{gh}$ in general, so for any $X \in G$, 
\[
CAC^{-1}XCA^{-1}C^{-1} = BXB^{-1}.
\]
This can be rewritten as
\[
(B^{-1}CAC^{-1}) X(CA^{-1}C^{-1}B) = X.
\]
In other words, $\Inn_{B^{-1} CAC^{-1}} = \Id$. By \cref{fixG}, $B^{-1}CAC^{-1} = p\Id$ for some $p \in k$. Thus 
\begin{equation}\label{CACinveqpB}
CAC^{-1} = pB
\end{equation}
\end{proof}

Denote by $k^*$ the multiplicative subgroup of $k$ and by $(k^*)^2$ the non-zero squares of $k$. We strengthen \cref{cacinv} as follows.

\begin{theorem}\label{thmAB}
If $\Inn_A$ and $\Inn_B \in \Aut(G,G_k)$ are involutions corresponding to $A = \begin{bsmallmatrix} 0 & 1\\ r & 0\end{bsmallmatrix}$ and $B = \begin{bsmallmatrix} 0 & 1\\ s & 0\end{bsmallmatrix}$, then $\Inn_A$ is isomorphic to $\Inn_B$ if and only if $r/s \in (k^*)^2$.
\end{theorem}

\begin{proof}
By \cref{cacinv}, $\Inn_A$ and $\Inn_B$ are isomorphic if and only if $A$ is conjugate to $pB$ where $p \in k$. Two matrices are conjugate if and only if both their characteristic and minimal polynomials are equal. In this case, $M_A(x) = C_A(x) = x^2 + r$ and $M_{pB}(x) = C_{pB}(x) = x^2 + p^2 s$. So $A$ and $B$ are conjugate if and only if $r = p^2 s$; in other words, exactly when $r/s \in (k^*)^2$.
\end{proof}

\begin{corollary}
The number of isomorphism classes of $k$-involutions of $G_k$ equals the number of square classes of $k^*$.
\end{corollary}

The isomorphism classes of $k$-involutions of $G_k$ depend on the elements in $(k^*)^2$. By \cref{serrelemma}, every non-zero element of the finite field $\bbf_q$ is a square.
\begin{corollary}\label{isoclasscor}
Let $k$ be a finite or algebraically closed field (of characteristic 2). There is one isomorphism class of $k$-involutions of $G_k$.
\end{corollary}

If $k$ is any infinite, non-algebraically closed field, then there are infinitely many non-squares in $k$. For example, when $k$ is the imperfect field $\bbf_q(x)$, for some indeterminate $x$, there are non-squares $x$, $x^3 + 1$, $x^5 + 1$, all of which are in different square classes.

\begin{corollary}
Let $k$ be any infinite non-algebraically closed field. There are infinitely many isomorphism classes of $k$-involutions of $G_k$.
\end{corollary}

Having described the $k$-involutions of $G_k$, we continue our analysis and describe the fixed-point groups and symmetric $k$-varieties. The fixed point group 
\[
H^p = \{x \in G_k \st \varphi(x) = x\}
\] of a $k$-involution of $G_k$ plays an important role in determining the structure of the corresponding symmetric $k$-variety $Q_k$.  For the involution $\Inn_A$, where $A = \begin{bsmallmatrix}0 & 1\\p & 0\end{bsmallmatrix}$, we have
\begin{align*}
H^p &= \big\{X \in G_k \st \Inn_A(X) = X\big\}\\
&= \left\{\begin{bmatrix} x_{11} & x_{12}\\px_{12} & x_{11}\end{bmatrix} \Bigg|\  x_{11}^2 + px_{12}^2 = 1\right\}.
\end{align*}

\begin{lemma}\label{isoH}
Let $\varphi_1$ and $\varphi_2$ be isomorphic $k$-involutions of $G_k$ with fixed point groups $H_1$ and $H_2$, respectively. Then $H_1$ and $H_2$ are isomorphic.
\end{lemma}

\begin{proof}
Let $\psi \in \Aut(G,G_k)$. If $\psi\varphi_1\psi^{-1} = \varphi_2$, then it suffices to show that $\psi(H_1) = H_2$. For $X_1 \in H_1$, $\psi(\varphi_1(X_1)) = \psi(X_1) = \varphi_2(\psi(X_1)) \in H_2$ since $\varphi_2 (\psi(X_1)) = \psi(X_1)$. Hence $\psi:H_1 \to H_2$ is the required isomorphism.
\end{proof}

Let $Y \in H_k$ as above; then $C_Y(x) = (x + 1)^2$, and $Y$ is diagonalizable if and only if $Y = \begin{bsmallmatrix}1 & 0 \\ 0 & 1\end{bsmallmatrix}$. So whenever $|k| > 2$, $H_k$ consists of non-semisimple elements. Moreover, we shall soon see these fixed point groups consist of unipotent elements.

\begin{example}
When $k$ does not have characteristic 2 and $A = \begin{bsmallmatrix}0&1\\p&0\end{bsmallmatrix}$, then $\Inn_A$ has fixed point group
\[
H^p = \left\{\begin{bmatrix}x & y\\py & x\end{bmatrix}\Bigg| x^2 - py^2 = 1\right\}.
\]
For $k = \bbr$, either $p = 1$ or $p = -1$. For $k = \bbc$, $p = 1$. Wu \cite{wu02} showed that $H^p$ is $k$-split if and only if $p$ is a square in $k^*$, and otherwise $H^p$ is $k$-anisotropic. Thus $H^p$ is $k$-anisotropic if $k = \bbr$, and $H^p$ is non-compact if $k = \bbc$. In this case, the fixed point groups of $k$-involutions corresponding to semisimple matrices are reductive.
\end{example}

\begin{theorem}
If $A = \begin{bsmallmatrix}0&1\\p&0\end{bsmallmatrix}$ and $k$ has one square class, then $\Inn_A$ is isomorphic to $\Inn_B$ where 
\[
B = \begin{bmatrix}1 & \frac{1}{\sqrt{p}}\\0 & 1\end{bmatrix}.
\]
Moreover, the fixed-point group $H^p$ of $\Inn_A$ is isomorphic to a unipotent subgroup 
\[
H^p \cong 
\left\{\begin{bmatrix}1 & x\\0 & 1\end{bmatrix} \Bigg|\ x \in k\right\},
\]
hence every $k$-involution of $G_k$ has a unipotent fixed-point subgroup.
\end{theorem}

\begin{proof}
Note that 
\[
\begin{bmatrix}1 & 0\\ \sqrt{p} & 1\end{bmatrix}
\begin{bmatrix}0 & 1\\ p & 0\end{bmatrix}
\begin{bmatrix}1 & 0\\ \sqrt{p} & 1\end{bmatrix}
= 
\begin{bmatrix}\sqrt{p} & 1\\0 & \sqrt{p}\end{bmatrix},
\]
so if $p \in (k^*)^2$, then $\Inn_A$ is isomorphic to $\Inn_B$ where 
\[
B = \begin{bmatrix}1 & \frac{1}{\sqrt{p}}\\0 & 1\end{bmatrix}. 
\]
The fixed-point group of $\Inn_B$ is
\[
\left\{\begin{bmatrix}x + \sqrt{p}y & y\\0 & x + \sqrt{p}y\end{bmatrix} \Bigg|\ x \in k\right\} \cong \left\{\begin{bmatrix}1 & x\\0 & 1\end{bmatrix} \Bigg|\ x \in k\right\}.
\]
\end{proof}

Since every $k$-involution of $G_k$ is isomorphic to conjugation by an element of the form $A = \begin{bsmallmatrix}0&1\\p&0\end{bsmallmatrix}$, for some $p \in k$, the elements of $Q_k \cong G_k/H_k$ have the following form:

\begin{align*}
Q_k &= \left\{X\theta(X)^{-1} \st X \in G\right\}\\
&= \left\{X\left[\Inn_A(X)\right]^{-1} \st X \in G\right\}\\
&= \left\{X\left[AXA^{-1}\right]^{-1} \st X \in G\right\}\\
&= \left\{XAX^{-1}A^{-1} \st X \in G\right\}\\
&= \left\{
\begin{bmatrix}
a & b \\
c & d
\end{bmatrix}
\begin{bmatrix}
0 & 1 \\
p & 0
\end{bmatrix}
\begin{bmatrix}
d & b \\
c & a
\end{bmatrix}
\begin{bmatrix}
0 & \frac{1}{p} \\
1 & 0
\end{bmatrix}
\Bigg|\ 
ad + bc = 1
\right\}\\
&= \left\{\begin{bmatrix}
pb^2 + a^2 & \frac{1}{p}(pbd + ac)\\
pbd + ac & \frac{1}{p}(pd^2 + c^2)
\end{bmatrix}
\Bigg|\ ad + bc = 1\right\}
\end{align*}

The minimal polynomial of any element $X \in Q_k$ has the form 
$$M_X(t) = t^2 + t\left(a^2 + d^2 + pb^2 + \frac{c^2}{p}\right) + \frac{1}{p}\left(pdb + ac\right)^2$$
which can be factored into distinct linear factors if and only if $a^2 + d^2 + pb^2 + \frac{c^2}{p} = 0$. That means that $X \in Q_k$ is semisimple exactly when $a^2 + d^2 + pb^2 + \frac{c^2}{p} \ne 0$.

\section{$k$-involutions of $\SL(n,k)$ for $n > 2$}\label{slnk}

In this section, we build upon the results from the previous section. As before, let $k$ be a field of characteristic 2, $K$ an extension field of $k$, $G = \SL(n,\ol{k})$, $G_K = \SL(n,K)$, and $G_k = \SL(n, k)$. 

Now that $n > 2$, there are both inner and outer automorphisms of $G_k$, and we must deal with each case separately. As before, begin by determining which of the inner automorphisms fix $G_k$ as a group and which fix $G_k$ point-wise.

\begin{lemma}\label{pId}
Suppose $A \in \GL(n,K)$. Then $\Inn_A|_{G_k} = \Id$ if and only if $A = p\Id$, for some $0 \ne p \in K$.
\end{lemma}

\begin{proof}
Following the same logic in the proof of \cref{fixG}, let $A \in \GL(n,K)$. By equating the $(i,j)$ entries of $AX$ and $XA$, we obtain the following equation:
\begin{equation}\label{AXeqXA}
\sum_{k = 1}^n a_{ik} x_{kj} = \sum_{\ell = 1}^n x_{i\ell}a_{\ell j}.
\end{equation}
Let $X$ be the matrix with 1 on the diagonal and in the $(r,s)$ position, and 0 elsewhere. For $i = r$ and $j = s$,\cref{AXeqXA} becomes
$$a_{r,s} + a_{r,r} = a_{s,s} + a_{r,s},$$
and hence $a_{r,r} = a_{s,s}$. By varying $r$ and $s$, it follows that $a_{i,i} = a_{j,j}$ for all $i$ and $j$. For $i = t$ and $j = s$ ($t \ne r, s$) \cref{AXeqXA} becomes
$$a_{t,r} + a_{t,s} = a_{t,s},$$
and hence $a_{t,r} = 0$. By varying $r,s,$ and $t$, it follows that $a_{i,j} = 0$ for all $i \ne j$.

As before, the converse is clear.
\end{proof}

\begin{lemma}\label{ApB}
Let $A \in \GL(n,K)$. Then $\Inn_A|_{G_k} \in \Aut(G_k)$ if and only if $A = pB$, for some $p \in K$ and $B \in \GL(n,k)$.
\end{lemma}

\begin{proof}
Note that
\begin{equation}\label{innAX}
\Inn_A(X) = AXA^{-1} = \frac{1}{\det(A)}AXD,
\end{equation}
where $D$ is the adjugate or adjoint matrix of $A$. That is, $D$ is the transpose of the cofactor matrix of $A$. So the $(i,j)$ entry of \cref{innAX} is
\begin{equation}\label{ijinnAX}
\frac{1}{\det A} \left( \sum_{m = 1}^n\sum_{\ell = 1}^n a_{i,m}x_{m,\ell}d_{j,\ell}\right),
\end{equation}
an element of $k$. Setting $X = \Id$ yields
\[
\frac{1}{\det A}\left( a_{i,1}d_{j,1} + \dots  + a_{i,n}d_{j,n}\right),
\]
and setting $X = \Id + E_{pq}$ (for any $p \ne q$) yields
\[
\frac{1}{\det A}\left( a_{i,1}d_{j,1} + \dots  + a_{i,n}d_{j,n} + a_{i,p}d_{q,j}\right).
\]
The sum of these two expressions is
\begin{equation}\label{Apiaqj}
\frac{a_{i,p}d_{q,j}}{\det(A)} \in k.
\end{equation}
Let $X$ be the permutation matrix with rows $p$ and $q$ switched. Then \cref{innAX} becomes
\[
\frac{1}{\det(A)}\left(a_{i,q}d_{p,j} + a_{i,p}d_{q,j} + \sum_{\ell = 1 \atop \ell \ne p,q}^na_{i,\ell}d_{j,\ell}\right).
\]
Finally, let $X$ be as above with $x_{q,q} = 1$ instead of 0. Then \cref{innAX} becomes
\[
\frac{1}{\det(A)}\left(a_{i,q}d_{p,j} + (a_{i,p} + a_{i,q})d_{q,j} + \sum_{\ell = 1 \atop \ell \ne p,q}^na_{i,\ell}d_{j,\ell}\right).
\]
Adding these last two expressions yields 
\begin{equation}\label{Aqiaqj}
\frac{a_{i,q}d_{q,j}}{\det(A)}  \in k.
\end{equation}

By dividing \cref{Apiaqj} by \cref{Aqiaqj} and re-labeling indices as needed, we have
$$
\frac{\dfrac{a_{i,j}d_{m,\ell}}{\det(A)}}{\dfrac{a_{m,n}d_{m,\ell}}{\det(A)}}
=
\frac{a_{i,j}}{a_{m,n}} \in k
$$
for any $m,n,i,j \in \{1, 2, \dots n\}$, provided that $a_{m,n} \ne 0$. Now, by factoring $a_{m,n}$ from $A$, we get $A = a_{m,n}B = pB$ where $b_{i,j} = \frac{a_{i,j}}{a_{m,n}}$, and $\Inn_A = \Inn_B$. The converse is again clear.
\end{proof}

\begin{corollary}\label{innAglnk}
Any inner automorphism of $G_k$ can be written as conjugation by a matrix in $\GL(n,k)$.
\end{corollary}

We recall \cref{cacinv} whose proof holds in the present situation.

\begin{lemma}\label{innAconjinnB}
Let $\Inn_A$ and $\Inn_B$ be automorphisms in $\Aut(G,G_k)$. Then $\Inn_A$ and $\Inn_B$ are isomorphic if and only if there is a matrix $C \in \GL(2,k)$ and a constant $p \in k$ such that $CAC^{-1} = pB$.
\end{lemma}

By \cref{pId}, if $\Inn_A \in \Aut(G,G_K)$ fixes $G_k$ point-wise, then $A = p\Id$ for some $p \in K$. By \cref{ApB}, we may take $p \in k$. Since $\Inn_A^2 = \Inn_{A^2}$, we must determine the restrictions on $A$ such that $A^2 = p\Id$. Recall that an involution is an automorphism of order exactly 2.

\begin{lemma}\label{asquaresln}
Suppose $A \in \GL(n,k)$ with $A^2 = p\Id$.
\begin{enumerate}
\item If $A = c\Id$ where $c^2 = p$ and $c \in k$, then $\Inn_A$ is not a $k$-involution.
\item If $A^2 + c^2\Id = 0$, where $c \in k$, and $A \ne c\Id$, then $A$ is conjugate to a matrix with $m$ copies of $L_{c^2}$ and $n - 2m$ copies of $c$ on the diagonal, as in
\begin{equation}\label{lmpc}
L_{m,c^2, c} = 
\begin{bmatrix}
\vspace{-.33cm}\\
\Lx{c^2} \\
& \ddots \\
& & \Lx{c^2}\\
& & & c \\
& & & & \ddots  \\
& & & & &  c\ 
\vspace{.035cm}
\end{bmatrix}.
\end{equation}
\item If $A^2 + p\Id = 0$, where $p \notin (k^*)^2$, then $A$ is conjugate to a matrix with $\frac{n}{2}$ copies of $L_p = \Lx{p}$ on the diagonal, as in
\begin{equation}\label{ln2p}
L_{\frac{n}{2},p} = 
\begin{bmatrix}
\vspace{-.33cm}\\
\Lx{p} & & \\
 & \ddots & \\
&  & \Lx{p}
\vspace{.1cm}
\end{bmatrix}.
\end{equation}
\end{enumerate}
\end{lemma}

\begin{proof}
~
\begin{enumerate}
\item[2.] Since the minimal polynomial of $A$ is $M_A(x) = x^2 + c^2 = (x + c)^2$, and the invariant factors of $A$ must divide $M_A(x)$, we know some of the structure of the rational canonical form of $A$. Without knowing $A$ explicitly we can not determine the multiplicities of the invariant factors, so there may be any combination of invariant factors of the form $(x + c)^2$ or $(x + c)$. The only constraint is that the sum of the degrees of the invariant factors equals $n$, and that there is at least one factor of $(x + c)^2$. Therefore $A = L_{m,c^2,c}$, where $m$ varies from $1$ to $n/2$.
\item[3.] The minimal polynomial is $M_A(x) = x^2 + p$, where $p$ is not a square. The characteristic polynomial is $C_A(x) = (x^2 + p)^{n/2}$, and here, the invariant factors are all $M_A(x)$, since they are forced to divide $M_A(x)$. Since the sum of the degrees of all the invariant factors must equal the degree of $C_A(x) = n$, the dimension of $A$ is even. Therefore, $A$ is a direct sum of $\frac{n}{2}$ blocks of $L_p$, as in $L_{\frac{n}{2},p}$. 
\end{enumerate}
Notice that the last case only occurs when $k$ is both infinite and {\em not} algebraically closed.
\end{proof}

\begin{corollary}\label{corcharInnA}
Suppose $\varphi \in \Aut(G, G_k)$ is an inner involution. There is a matrix $A \in \GL(n,k)$ such that $\varphi  = \Inn_A$ and $A$ is conjugate to one of the following:
\begin{enumerate}
\item $L_{m,c^2,c}$ for $c \in k^*$.
\item $L_{\frac{n}{2}, p}$ for some $p \in k^*/(k^*)^2$.
\end{enumerate}
\end{corollary}

\begin{proof}
By \cref{innAglnk}, there is a matrix $A \in \GL(n,k)$ so that $\varphi = \Inn_A$. Since $\varphi$ is an involution, we have $\varphi^2 = \Inn_{A^2} = \Id$. Moreover, by \cref{pId}, $A^2 = p\Id$ for some $p \in k$. It follows from \cref{asquaresln} that $A$ is conjugate to either $L_{\frac{n}{2},p}$ or $L_{m,c^2,c}$.
\end{proof}

Notice the special case when $L_{m,c^2,c}$ and $L_{\frac{n}{2},p}$ appear to be the same; this occurs when there are $n/2$ blocks of $L_{c^2}$ as in $L_{\frac{n}{2}, c^2, c}$, for $n$ even. The difference in $L_{\frac{n}{2},p}$ is that $p \notin (k^*)^2$, whereas in $L_{m,c^2,c}$, $c^2\in (k^*)^2$, so $c \in k$. This distinction partially determines isomorphism classes of $k$-involutions.

Also note that $L_{m,c^2,c}$ and $L_{\frac{n}{2},p}$ have different minimal and characteristic polynomials and different invariant factors, so they are not conjugate. Moreover, whenever $m_1 \ne m_2$, by the same argument $L_{m_1, c^2, c}$ is not conjugate to $L_{m_2,c^2,c}$. This is because the number of blocks of $L_{c^2}$ is different, which means they have different numbers of invariant factors.

Our next step is to determine when $L_{m,c^2,c}$ and $L_{m,d^2,d}$ correspond to isomorphic $k$-involutions and when $L_{\frac{n}{2},p}$ and $L_{\frac{n}{2},q}$ correspond to isomorphic $k$-involutions.

\begin{lemma}\label{charln2p}
$\Inn_{L_{\frac{n}{2}, p}}$ is isomorphic to $\Inn_{L_{\frac{n}{2}, q}}$ if and only if $p/q \in (k^*)^2$.
\end{lemma}

\begin{proof}
By \cref{ln2p}, $p$ and $q \in k^*/(k^*)^2$, and by \cref{innAconjinnB}, $\Inn_{L_{\frac{n}{2}, p}}$ is isomorphic to $\Inn_{L_{\frac{n}{2}, q}}$ if and only if $A = L_{\frac{n}{2},p}$ is conjugate to $cB = cL_{\frac{n}{2},q}$ for some $c \in k$. The minimal polynomials are $M_A(x) = x^2 + p$ and $M_{cB}(x) = x^2 + c^2q$, and  the characteristic polynomials are $C_A(x) = (x^2 + p)^{n/2}$ and $C_B(x) = (x^2 + c^2q)^{n/2}$, respectively. This means that $A$ and $cB$ are conjugate if and only if $c^2q = p$, since the invariant factors are fixed by the minimal polynomials. But this implies that $c^2 = p/q$. Thus $p/q \in (k^*)^2$.
\end{proof}

\begin{lemma}\label{charlmc}
Let $b, c \in k^*$. Then $\Inn_{L_{m,b^2,b}}$ is isomorphic to $\Inn_{L_{m,c^2,c}}$. 
\end{lemma}

\begin{proof}
By \cref{innAconjinnB}, $\Inn_{L_{m,b^2,b}}$ is isomorphic to $\Inn_{L_{m,c^2,c}}$ if and only if  $A = L_{m,b^2,b} \cong tB = tL_{m,c^2,c}$ for some $t \in k$. By \cref{pId}, the minimal polynomials are $M_A(x) = (x^2 + b^2)(x + b)$ and $M_{tB}(x) = (x^2 + t^2c^2)(x + tc)$. If $A$ and $tB$ are conjugate, their invariant factors, characteristic polynomials, and minimal polynomials are equal. So they are conjugate if and only if $tc = b$ and $t^2c^2 = b^2$. By setting $t = b/c$, it follows that $\Inn_{L_{m,b^2,b}}$ is isomorphic to $\Inn_{L_{m,c^2,c}}$.
\end{proof}

\begin{theorem}\label{innerslnkclasses}
Suppose $\varphi \in \Aut(G, G_k)$ is an inner $k$-involution. Then, up to isomorphism, $\varphi$ is of the form 
\begin{enumerate}
\item $\Inn_A$ where $A = L_{\frac{n}{2},p}$ for some $p \in k^*/(k^*)^2$.
\item $\Inn_A$ where $A = L_{m, c^2, c}$ for some $c \in k^*$.
\end{enumerate} 
\end{theorem}

\begin{proof}
Immediate from \cref{charln2p}, \cref{charlmc}, and \cref{corcharInnA}.
\end{proof}

As when $n=2$, the isomorphism classes of $k$-involutions of $\SL(n,k)$ depend on the number of square classes in $k^*$. All inner $k$-involutions of $G_k$ correspond to matrices which are not semisimple. However, there exist unipotent elements which correspond to inner $k$-involutions, as when $n = 2$.

When $k$ does not have characteristic 2, there are several differences as compared to the characteristic not 2 case. First, the matrices $I_{n-i,i}$ are diagonal matrices which correspond to $k$-involutions, so there are semisimple elements which correspond to $k$-involutions whenever $k$ is perfect with characteristic not 2. 

The $k$-involution $\Inn_{L_{m,c^2,c}}$ is unique to fields of characteristic 2. By \cref{innerslnkclasses}, whenever $k$ is algebraically closed or finite (and all field elements have unique squares), the $k$-involutions may be represented by either $\Inn_{L_{\frac{n}{2},1}}$ or $\Inn_{L_{m,1,1}}$. In summary, we have the following corollary to \cref{innerslnkclasses}, analogous to the corollary to the $n=2$ case.

\begin{corollary}
~
\begin{enumerate}
\item If $k$ is any finite field or algebraically closed field of characteristic 2 there is one isomorphism class of inner $k$-involutions of $\SL(n,k)$ corresponding to $\lnp$, and there are $\lfloor\frac{n}{2}\rfloor$ isomorphism classes corresponding to $\lmc$.
\item Otherwise, there are infinitely many isomorphism classes of inner $k$-involutions of $\SL(n,k)$ corresponding to $\lnp$, and there are  $\lfloor\frac{n}{2}\rfloor$ isomorphism classes corresponding to $\lmc$.
\end{enumerate}
\end{corollary}

For $n > 2$, there are $k$-automorphisms of $G_k$ which are not inner. The results in this section are similar to those obtained by Wu in his Ph.D. thesis \cite{wu02}. We restate several results and refer to Wu for proofs, except when different proofs are needed due to the characteristic of $k$.

In this section, let $\theta(X) = (X^{-1})^T$. Note that $\theta$ is not an inner automorphism unless $n = 2$, in which case $\theta = \Inn_X$ where $X =\Lx{1}$. It is known that for $n > 2$, $\Inn(G, G_k)$ has index 2 in $\Aut(G,G_k)$. Hence we can write any outer automorphism of $G_k$ as $\theta \Inn_A$ for some $A \in \GL(n,k)$. For convenience, we choose $\theta$ as above; note that $\theta \in \Aut(G,G_k)$ has order 2, and thus $\theta$ is an involution. Our first step is to determine which of the automorphisms $\theta\Inn_A$ square to the identity.

\begin{lemma}[Wu \cite{wu02}, Lemma 20]\label{symmouter}
$\theta \Inn_A$ is an involution if and only if $A$ is symmetric.
\end{lemma}

\begin{lemma}[Wu \cite{wu02}, Lemma 21]\label{outcong}
$\theta \Inn_A$ is isomorphic to $\theta \Inn_B$ if and only if $A$ is congruent to $pB$, for some $p \in k$. 
\end{lemma}

Non-singular symmetric matrices are congruent to diagonal matrices if there is at least one non-zero diagonal entry. When all diagonal entries are zero, the matrix is called an alternate matrix. Note that when $k$ is not of characteristic two, this concept is equivalent to the matrix being skew-symmetric. Alternate matrices correspond to symplectic transformations, and non-singular alternate matrices form one congruence class. In summary,

\begin{lemma}\label{outclass}
Let $A$ be a symmetric, non-singular matrix over a field of characteristic 2. Then 
\begin{enumerate}
\item If $A$ is alternate, then $A$ is congruent to $J = \begin{bsmallmatrix}0 & I\\I & 0\end{bsmallmatrix}$, and $A$ has even dimension.
\item If $A$ is not alternate, then $A$ is congruent to a diagonal matrix.
\item Let $b_i \in k^*$ and $a_i \in k$, and let
\[
A =
\begin{bmatrix}
a_1 & &\\
& \ddots &\\
&& a_n
\end{bmatrix}
\text{ and }
B = 
\begin{bmatrix}
b_1^2a_1 & &\\
& \ddots &\\
&& b_n^2a_n
\end{bmatrix}.
\]
Then $A$ is congruent to $B$.
\end{enumerate}
\end{lemma}

\begin{proof}
~
\begin{enumerate}
\item See Roman \cite{rom08}. %
\item Albert \cite{alb38}. 
\item Set $Q = \diag(b_1, \dots, b_n)$, and notice that $Q^TAQ = B$.
\end{enumerate}
\end{proof}

In characteristic 2, isomorphism classes of outer involutions are somewhat simpler. To summarize our work thus far, outer involutions can always be represented as $\theta\Inn_A$ where $A$ is either a diagonal matrix or the matrix $J$. We can always write $A$ as $\diag(1, \dots, 1, N_{p_1}, \dots, N_{p_r})$ where $N_{p_i}$ are (possibly distinct) non-squares in $k$. In this case, whenever $k$ has only one square class $A = \Id$.  The following theorem summarizes the isomorphism classes.

\begin{theorem}\label{slnkchar}
Let $\theta \Inn_A$ be an (outer) $k$-involution.
\begin{enumerate}
\item If $k$ is finite, or if $k$ is algebraically closed, there is one isomorphism class of outer $k$-involutions corresponding to symmetric non-alternate matrices. There is also one isomorphism class of outer $k$-involutions corresponding to alternate matrices.
\item If $k$ is infinite and not algebraically closed, there are infinitely many classes of outer $k$-involutions corresponding to symmetric non-alternate matrices. There is one isomorphism class of outer $k$-involutions corresponding to alternate matrices.
\end{enumerate}
\end{theorem}

We now turn to fixed point groups.

\begin{lemma}
~
\begin{enumerate}
\item The fixed point group of the involution $\Inn_{L_{m,c^2,c}}$ consists of matrices in $G_k$ of the form
$$
\begin{bmatrix}
\vspace{-.4cm}\\
\begin{bsmallmatrix}a_{1,1} & b_{1, 1}\\ c^2b_{1,1} & a_{1,1}\end{bsmallmatrix} & \dots & \begin{bsmallmatrix}a_{1,2m-1} & b_{1,2m-1}\\c^2b_{1,2m-1} & a_{1,2m-1}\end{bsmallmatrix} & \begin{smallmatrix}h_{1, 2m} & \dots & h_{1, n}\\ch_{1, 2m} & \dots & ch_{1, n}\end{smallmatrix}\\
\vdots & &\vdots & \vdots\\
\begin{bsmallmatrix}a_{2m,1} & b_{2m, 1} \\c^2b_{2m-1,1} & a_{2m-1,1} \end{bsmallmatrix} & \dots & \begin{bsmallmatrix}a_{2m-1,2m-1} & b_{2m-1,2m-1}\\c^2b_{2m-1,2m-1} & a_{2m-1,2m-1}\end{bsmallmatrix} & \begin{smallmatrix}h_{2m, 2m} & \dots & h_{2m, n}\\ch_{2m, 2m} & \dots & ch_{2m, n}\end{smallmatrix}\\\\
\begin{smallmatrix} g_{2m,1} & cg_{2m, 1} \\\vdots & \vdots\\g_{n,1} & cg_{n,1} \end{smallmatrix} & \dots &
\begin{smallmatrix} g_{2m,2m-1} & cg_{2m,2m-1} \\ \vdots & \vdots\\g_{n,2m-1} & cg_{n,2m-1} \end{smallmatrix} &
\begin{smallmatrix} s_{2m, 2m} & \dots & s_{2m, n}\\  \vdots & & \vdots\\ s_{n, 2m} & \dots & s_{n, n}\end{smallmatrix}
\end{bmatrix}
$$
which has $m$ $2 \times 2$ blocks in the upper left-hand corner. The upper right-hand corner has the property that the even rows are $c$-multiples of the preceding odd rows. Analogously, the bottom left-hand corner has the property that the even columns are $c$-multiples of the preceding odd columns. Finally, in the bottom right-hand corner, there are no relations.
\item The fixed point group of the involution $\Inn(L_{\frac{n}{2}, p})$ consists of the block matrices 
$$
\begin{bmatrix}
\vspace{-.4cm}\\
\begin{bmatrix} a_{1,1} & b_{1,1}\\pb_{1,1} & a_{1,1}\end{bmatrix} & \dots & \begin{bmatrix} a_{1,n-1} & b_{1,n-1}\\pb_{1,n-1} & a_{1,n-1}\end{bmatrix}\\
\vdots & & \vdots\\
\begin{bmatrix} a_{n-1,1} & b_{n-1,1}\\pb_{n-1,1} & a_{n-1,1}\end{bmatrix} & \dots & \begin{bmatrix} a_{n-1,n-1} & b_{n-1,n-1}\\pb_{n-1,n-1} & a_{n-1,n-1}\end{bmatrix}\\
\vspace{-.4cm}
\end{bmatrix}.
$$
This matrix consists of $2\times 2$ blocks, and is basically the same as in case (1) above, in the upper left-hand corner.
\end{enumerate}
\end{lemma}

\begin{proof}
~
\begin{enumerate}
\item The matrix $L_{m,c^2,c}$ acts on a matrix $A$ when multiplying on the left and on the right. Let $R_i$ denote the $i$th row of $A$, and $C_i$ denote the $i$th column of $A$. Multiplication on the left by $L_{m,c^2,c}$ changes $A$ as follows:
\begin{itemize} 
\item Replace $R_i$ by $c^2R_{i-1}$ when $i$ is even and $1 < i \le 2m$. 
\item Replace $R_j$ by $R_{j + 1}$ when $j$ is odd and $1 \le j < 2m$.
\item Replace $R_k$ by $cR_k$ for $k > 2m$.
\end{itemize}
Multiplication on the right by $L_{m,c^2,c}^{-1}$ changes $A$ as follows:
\begin{itemize} 
\item Replace $C_i$ by $\frac{1}{c^2}C_{i-1}$ when $i$ is even and $1 < i \le 2m$. 
\item Replace $C_j$ by $C_{j + 1}$ when $j$ is odd and $1 \le j < 2m$.
\item Replace $C_k$ by $\frac{1}{c}C_k$ for $k > 2m$.
\end{itemize}
Now equate entries to get the desired relations. 
\item This is simpler than case 1, and the result corresponds to the upper left-hand corner of the matrix in case 1.
\end{enumerate}
\end{proof}

The matrix $L_{m,c^2,c}$ is conjugate to a constant times a unipotent matrix. Let  
$$
U_{m,c} = 
\begin{bmatrix}
\vspace{-.43cm}\\
\Ux{c}\\
& \ddots\\
& & \Ux{c}\\
& & & 1\\
& & & & \ddots\\
& & & & & 1
\end{bmatrix}.
$$
Then $U_{m,c} L_{m,c^2,c} U_{m,c} = cB$ where
\begin{equation}\label{B1}
B = 
\begin{bmatrix}
\vspace{-.43cm}\\
1 & c^{-1}\\
& \ddots & \ddots\\
& & 1 & c^{-1}\\
& & & 1\\
& & & & \ddots\\
& & & & & 1
\end{bmatrix}.
\end{equation}
By a similar method, $L_{\frac{n}{2},p}$ is conjugate (over $G_K$) to a matrix of the form $\sqrt{p}B$ where
\begin{equation}\label{B2}
B = \begin{bmatrix}
1 & \sqrt{p}^{-1}\\
& \ddots & \ddots\\
& & \ddots & \sqrt{p}^{-1}\\
& & & 1
\end{bmatrix}.
\end{equation}
In both cases, the matrices of $k$-involutions are conjugate to unipotent matrices. Over fields of characteristic not 2, inner 
involutions are always semisimple. 


\begin{lemma}
Let $\varphi = \theta\Inn_A$ as in \cref{outclass}. Then
\begin{enumerate}
\item If $A$ is symmetric and not alternate and $|k^*/(k*)^2| = 1$, then the fixed-point group of $\varphi$ is $\Orth(n,k)$.
\item If $A$ is symmetric and not alternate and $|k^*/(k*)^2| \ge 2$, then the fixed-point group of $\varphi$ is the set 
\[
\{X \in G_k\ |\ A = X^TAX\}.
\]
\item If $A$ is symmetric and alternate, then the fixed point group of $\varphi$ is the set
\[
\{X \in G_k\ |\ X = JXJ\}.
\]
In other words, the fixed-point group consists of the block-symmetric matrices
\[
X = \begin{bmatrix} X_1 & X_2\\X_2 & X_1\end{bmatrix}.
\]
\end{enumerate}
\end{lemma}

\begin{proof}
~
\begin{enumerate}
\item By \cref{outclass}, $A$ is diagonal, and we may as well take $A = \Id$ as a representative. Thus $\theta\Inn_A(X) =  \theta(X) = (X^{-1})^T = X$, so $X \in \Orth(n,k)$.
\item Since $\theta\Inn_A(X) = \theta(AXA^{-1}) = X$, then $X = A^{-1}(X^{-1})^TA$, and $AX^{-1} = X^TA$.
\item Write $J = \begin{bsmallmatrix}0 & I\\ I & 0\end{bsmallmatrix}$ and $X = \begin{bsmallmatrix}X_1 & X_2\\ X_3 & X_4\end{bsmallmatrix}$. Then $\Inn_J(X) = JXJ^{-1} = JXJ = X$.
\end{enumerate}
\end{proof}

For the involutions $\Inn_{L_{m,c^2,c}}$ and $\Inn_{L_{\frac{n}{2},p}}$, $Q_k$ has the following structure:
\begin{align*}
Q_k &= \left\{X \big(\Inn_A(X)\big)^{-1} \st X \in G_k\right\}\\
&= \left\{X(A^T)^{-1}X^TA^T \st X \in G_k\right\}.
\end{align*}
In the case that the involutions are outer, since $A = A^T$, we have
\begin{align*}
Q_k &= \left\{X \big(\theta\Inn_A(X)\big)^{-1} \st X \in G_k\right\}\\
&= \left\{X\theta(AXA^{-1})^{-1} \st X \in G_k\right\}\\
&= \left\{X\big((A^T)^{-1}(X^T)^{-1}A^T   \big)^{-1} \st X \in G_k\right\}\\
&= \left\{X A^TX^T(A^T)^{-1} \st X \in G_k\right\}\\
&= \left\{X AX^TA^{-1} \st X \in G_k\right\}
\end{align*}
In the case that $A = \Id$,  $Q_k$ is the set of symmetric matrices in $G_k$. 

\bibliographystyle{plain}
\bibliography{slnk}

\end{document}